%% file: sfmotives102011.tex
\documentclass[a4paper,10pt]{amsart} 
\DeclareMathAlphabet{\mathbbb}{U}{bbold}{m}{n}

\usepackage{amsmath}
\usepackage{amsthm}
\usepackage{amssymb} 

\usepackage{diagrams}

\theoremstyle{plain}
\newtheorem{Theorem}{Theorem}[section]

\newtheorem{Proposition}[Theorem]{Proposition}

\newtheorem{Corollary}[Theorem]{Corollary}

\newtheorem{Lemma}[Theorem]{Lemma}

\newtheorem{Conjecture}[Theorem]{Conjecture}

\theoremstyle{definition}

\newtheorem{Definition}[Theorem]{Definition}

\theoremstyle{remark}

\newtheorem{Remark}[Theorem]{Remark}

\newtheorem{Scholium}[Theorem]{Scholium}
\newtheorem{Example}[Theorem]{Example}

\newtheorem{substuff}{\bf Remark}[Theorem]
\newtheorem{subrem}[substuff]{\bf Remark} 

\input{macroMotives}

\newcommand\chose[2]{\genfrac{(}{)}{0pt}{}{#1}{#2}}

\newcommand{\Sym}{\mathrm{Sym}}
\newcommand{\Z}{\mathbb{Z}}

\newcommand\eff{\text{eff}}
\newcommand\Meff{\mathcal{M}^\eff} 
\begin{document}

\title{Schur-finiteness in $\lambda$-rings}
\date{\today}

\author{C. Mazza}
\address{DIMA -- Universit\`a di Genova, Genova, Italy}
\email{mazza@dima.unige.it} 
\author{C. Weibel}
\thanks{Weibel's research was supported by NSA and NSF grants.}
\address{Dept.\ of Mathematics, Rutgers University, New Brunswick,
NJ 08901, USA} \email{weibel@math.rutgers.edu}

\begin{abstract}
We introduce the notion of a Schur-finite element in a $\lambda$-ring.
\end{abstract}

\maketitle


Since the beginning of algebraic $K$-theory in \cite{G57},
the splitting principle has proven invaluable for working with
$\lambda$-operations. Unfortunately, this principle does not seem
to hold in some recent applications, such as the $K$-theory of motives.
The main goal of this paper is to introduce the subring of
Schur-finite elements of any $\lambda$-ring, 
and study its main properties, especially in connection with
the virtual splitting principle.

A rich source of examples comes from Heinloth's theorem \cite{Heinloth}, that 
the Grothendieck group $K_0(\calA)$ of an idempotent-complete $\QQ$-linear 
tensor category $\calA$ is a $\lambda$-ring. For the category
$\Meff$ of effective Chow motives, we show that $K_0(Var)\to K_0(\Meff)$
is not an injection, answering a question of Grothendieck.

When $\calA$ is the derived category of motives $\DM_{gm}$ over a field of
characteristic~0, the notion of Schur-finiteness in $K_0(\DM_{gm})$ 
is compatible with the 
notion of a Schur-finite object in $\DM_{gm}$, introduced in \cite{Mazza}.

We begin by briefly recalling the classical splitting principle in
Section~1, and answering Grothendieck's question in Section~2.
In section~3 we recall the Schur polynomials, the Jacobi-Trudi identities
and the Pieri rule from the theory of symmetric functions.
Finally, in Section~4, we define Schur-finite elements and show that
they form a subring of any $\lambda$-ring. We also state the
conjecture that every Schur-finite element is a virtual sum of line elements.

\subsection*{Notation}

We will use the term $\lambda$-ring in the sense of \cite[2.4]{SGA6};
we warn the reader that our $\lambda$-rings are called 
{\it special $\lambda$-rings} by Grothendieck, Atiyah 
and others; see \cite{G57} \cite{AT} \cite{Atiyah}.

A $\QQ$-linear category $\calA$ is a category in which each
hom-set is uniquely divisible (i.e., a $\QQ$-module). By a $\QQ$-linear
tensor category (or  $\QQ$TC) we mean a $\QQ$-linear category which is also
symmetric monoidal and such that the tensor product is $\QQ$-linear.
We will be interested in $\QQ$TC's which are idempotent-complete.

\newpage

\section{Finite-dimensional $\lambda$-rings}

Almost all $\lambda$-rings of historical interest are
finite-dimensional. This includes the complex representation rings
$R(G)$ and topological $K$-theory of compact spaces \cite[1.5]{AT}
as well as the algebraic $K$-theory of algebraic varieties \cite{G57}.
In this section we present this theory from the viewpoint we are adopting.
Little in this section is new.

Recall that an element $x$ in a $\lambda$-ring $R$ is said to be
{\it even} of finite degree $n$ if $\lambda_t(x)$ is a 
polynomial of degree $n$, or equivalently that there is a
$\lambda$-ring homomorphism from the ring $\Lambda_n$ defined in
\ref{free-even} to $R$, sending $a$ to $x$.
We say that $x$ is a {\it line element} if it is even of degree~1,
i.e., if $\lambda^n(x)=0$ for all $n>1$.

We say that $x$ is {\it odd} of degree $n$ if
$\sigma_t(x)=\lambda_{-t}(x)^{-1}$ is a polynomial of finite degree
$n$.  Since $\sigma_{-t}(x)=\lambda_t(-x)$, we see that $x$ is odd
just in case $-x$ is even. Therefore there is a $\lambda$-ring
homomorphism from the ring $\Lambda_{-n}$ defined in \ref{free-even}
to $R$ sending $b$ to $x$.

We say that an element $x$ is {\it finite-dimensional}
if it is the difference of two even elements, or equivalently if
$x$ is the sum of an even and an odd element. The subset of
even elements in $R$ is closed under addition and multiplication, and the
subset of finite-dimensional elements forms a subring of $R$.

\begin{Example}\label{binom}
If $R$ is a binomial $\lambda$-ring, then
$r$ is even if and only if some $r(r-1)\cdots(r-n)=0$, and odd if and only if some
$r(r+1)\cdots(r+n)=0$. The binomial rings $\prod_{i=1}^n\Z$ are
finite dimensional. If $R$ is connected then the 
subring of finite-dimensional elements is just $\ZZ$.
\end{Example}

There is a well known family of universal finite-dimensional $\lambda$-rings
$\{\Lambda_n\}$.

\begin{Definition}\label{free-even}
Following \cite{AT}, let $\Lambda_n$ denote the free $\lambda$-ring 
generated by one element
$a=a_1$ of finite degree $n$ (i.e., subject to the relations that 
$\lambda^k(a)=0$ for all $k>n$). By \cite[4.9]{SGA6}, $\Lambda_n$ is 
just the polynomial ring $\Z[a_1,...,a_n]$ with $a_i=\lambda^i(a_1)$. 

Similarly, we write $\Lambda_{-n}$ for the free $\lambda$-ring 
generated by one element $b=b_1$,
subject to the relations that $\sigma^k(b)=0$ for all $k>n$.
Using the antipode $S$, we see that there is a $\lambda$-ring
isomorphism $\Lambda_{-n}\cong\Lambda_n$ sending $b$ to $-a$,
and hence that $\Lambda_{-n}\cong\ZZ[b_1,...,b_n]$ with
$b_k=\sigma^k(b)$.

Consider finite-dimensional elements in $\lambda$-rings $R$ which
are the difference of an even element of degree $m$ and an odd element
of degree $n$. The maps $\Lambda_m\to R$ and $\Lambda_{-n}\to R$
induce a $\lambda$-ring map from $\Lambda_m\otimes\Lambda_{-n}$ to $R$.
\end{Definition}

\begin{Lemma}
If an element $x$ is both even and odd in a $\lambda$-ring,
then $x$ and all the $\lambda^i(x)$ are nilpotent.
Thus $\lambda_t(x)$ is a unit of $R[t]$.
\end{Lemma}

\begin{proof}
If $x$ is even and odd then $\lambda_{t}(x)$ and $\sigma_{-t}(x)$
are polynomials in $R[t]$ which are inverse to each other.
It follows that the coefficients $\lambda^i(x)$ of the $t^i$
are nilpotent for all $i>0$.
\end{proof}

If $R$ is a graded $\lambda$-ring, an element $\sum r_i$ is even
(resp., odd, resp., finite-dimensional) if and only if each homogeneous
term $r_i$ is even (resp., odd, resp., finite-dimensional). This is
because the operations $\lambda^n$ multiply the degree of an element by $n$.


The forgetful functor from $\lambda$-rings to commutative rings
has a right adjoint; see \cite[pp.\,20--21]{Knutson}. It follows that
the category of $\lambda$-rings has all colimits. In particular,
if $B \leftarrow A \to C$ is a diagram of $\lambda$-rings, the tensor
product $B\otimes_AC$ has the structure of a $\lambda$-ring.
Here is a typical, classical application of this construction,
originally proven in \cite[6.1]{AT}.

\begin{Proposition}[Splitting Principle]\label{finite-splitting}
If $x$ is any even element of finite degree $n$ in a $\lambda$-ring $R$,
there exists an inclusion $R\subseteq R'$ of $\lambda$-rings and
line elements $\ell_1,...,\ell_n$ in $R'$ so that $x=\sum\ell_i$.
\end{Proposition}

\begin{proof}
Let $\Omega_n$ denote the tensor product of $n$ copies of the $\lambda$-ring
$\Lambda_1=\ZZ[\ell]$; this is a $\lambda$-ring whose underlying ring is
the polynomial ring $\ZZ[\ell_1,...,\ell_n]$, and the $\lambda$-ring
$\Lambda_n$ of Definition \ref{free-even} is the subring of
symmetric polynomials in $\Omega_n$; see \cite[\S2]{AT}.
Let $R'$ be the pushout of the diagram $\Omega_n \leftarrow \Lambda_n\to R$.
Since the image of $x$ is $1\otimes x=a\otimes1=(\sum\ell_i)\otimes1$,
it suffices to show that $R\to R'$ is an injection. This follows from
the fact that $\Omega_n$ is free as a $\Lambda_n$-module. 
\end{proof}

\begin{Corollary}\label{line-decompose}
If $x$ is any finite-dimensional element of a $\lambda$-ring $R$,
there is an inclusion $R\subseteq R'$ of $\lambda$-rings and
line elements $\ell_i$, $\ell'_j$ in $R'$ so that
\[ x=(\sum\ell_i) - (\sum\ell'_j). \]
\end{Corollary}

\begin{Scholium}\label{scholium}
For later use, we record an observation, whose proof is implicit in the 
proof of Proposition 4.2 of \cite{AT}: 
$\lambda^m(\lambda^n x)=P_{m,n}(\lambda^1 x,\dots,\lambda^{mn} x)$
is a sum of monomials, each containing a term $\lambda^i x$ for $i\ge n$.
For example, 
$\lambda^2(\lambda^3 x)=\lambda^6x-x\,\lambda^5x+\lambda^4x\,\lambda^2x$
(see \cite[p.\,11]{Knutson}).
\end{Scholium}

\section{$K_0$ of tensor categories}

The Grothendieck group of a $\QQ$-linear tensor category provides
numerous examples of $\lambda$-rings, and forms the original motivation
for introducing the notion of Schur-finite elements in a $\lambda$-ring.

A $\QQ$-linear tensor category is {\it exact} if it has a distinguished
family of sequences, called {\it short exact sequences} and satisfying the 
axioms of \cite{Quillen}, and such that each $A\otimes-$ is an exact functor.
In many applications $\calA$ is {\it split exact}: the only short exact
sequences are those which split.
By $K_0(\calA)$ we mean the Grothendieck group as an exact category,
i.e., the quotient of the free abelian group on the objects $[A]$ by
the relation that $[B]=[A]+[C]$ for every short exact sequence
$0\to A\to B\to C\to0$.

Let $\calA$ be an idempotent-complete exact category which is a 
$\QQ$TC for $\otimes$. 
%
For any object $A$ in $\calA$, the symmetric group
$\Sigma_n$ (and hence the group ring $\QQ[\Sigma_n]$) acts on the
$n$-fold tensor product $A^{\otimes n}$\!. If $\calA$ is idempotent-complete,
we define $\wedge^nA$ to be the direct summand of $A^{\otimes n}$ 
corresponding to the alternating idempotent 
$\sum (-1)^\sigma \sigma/n!$ of $\QQ[\Sigma_n]$.
Similarly, we can define the symmetric powers $\Sym^n(A)$.
It turns out that 
$\lambda^n(A)$ only depends upon the element $[A]$ in $K_0(\calA)$, and
that $\lambda^n$ extends to a well defined operation on $K_0(\calA)$.

The following result was proven by F.\,Heinloth in \cite[Lemma 4.1]{Heinloth}, 
but the result seems to have been in the air; see \cite[p.\,486]{Davidov},
\cite[5.1]{LL04} and \cite{B1, B2}.
A special case of this result was proven long ago by Swan in \cite{Swan}.

\begin{Theorem}\label{lambda-ring}
If $\calA$ is any idempotent-complete exact $\QQ$TC, 
$K_0(\calA)$ has the structure of a $\lambda$-ring. 
If $A$ is any object of $\calA$ then $\lambda^n([A])=[\wedge^nA]$.
\end{Theorem}

Kimura \cite{Kimura} and O'Sullivan have introduced the notion
of an object $C$ being finite-dimensional in any $\QQ$TC $\calA$: 
$C$ is the direct sum of an even object $A$ (one for which some 
$\wedge^nA\cong0$) and an odd object $B$ (one for which some 
$\Sym^n(B)\cong0$). It is immediate that $[C]$ is a finite-dimensional
element in the $\lambda$-ring $K_0(\calA)$. Thus the two notions
of finite dimensionality are related.

\begin{Example}\label{effective}
Let $\Meff$ denote the category of $\Q$-linear 
pure effective Chow motives with respect
to rational equivalence over a field $k$. Its objects are summands
of smooth projective varieties over a field $k$ and morphisms are given by
Chow groups. Thus $K_0(\Meff)$ is the group generated by the classes of
objects, modulo the relation $[M_1\oplus M_2]=[M_1]+[M_2]$.
Since $\Meff$ is a $\QQ$TC, $K_0(\Meff)$ is a $\lambda$-ring.

By adjoining an inverse to the Lefschetz motive to $\Meff$, we obtain the
category $\calM$ of Chow motives (with respect to rational equivalence).
This is also a $\QQ$TC, so $K_0(\calM)$ is a $\lambda$-ring.

The category $\Meff$ embeds into the triangulated category $\DM^{\eff}_{gm}$
of effective geometric motives; see \cite[20.1]{MVW}. Similarly,
the category $\calM$ embeds in the triangulated category $\DM_{gm}$ of
geometric motives \cite[20.2]{MVW}. Bondarko proved in 
\cite[6.4.3]{Bondarko} that $K_0(\DM^\eff_{gm})\cong K_0(\Meff)$
and $K_0(\DM_{gm})\cong K_0(\calM)$. Thus we may investigate
$\lambda$-ring questions in these triangulated settings.
As far as we know, it is possible that every element of 
$K_0(\DM_{gm})$ is finite-dimensional.
\end{Example}

Recall that a motive $M$ in $\Meff$ is a \textbf{phantom motive} 
if $H^*(M)=0$ for every Weil cohomology $H$.

\begin{Proposition}\label{m0phantom}
Let $M$ be an object of $\Meff$. Then if $[M]=0$ in $K_0(\Meff)$, then
$M$ is a phantom motive.
\end{Proposition}

\begin{proof}
  Since $\Meff$ is an additive category, $[M]=0$ implies that
there is another object $N$ of $\Meff$ such that $M\oplus N\isom N$.
But every effective motive is a summand of the motive of a 
scheme, hence we may assume $N=M(X)$. If $M$ is not a phantom motive, there
is a Weil cohomology and an $i$ such that $H^i(M)\not=0$. But then
$H^i(M)\oplus H^i(X)\isom H^i(X)$; since these are finite-dimensional
vector spaces, this implies $H^i(M)=0$, a contradiction.
\end{proof}

Here is an application of these ideas. Recall that any quasi-projective 
scheme $X$ has a motive with compact supports in $\DM^\eff$, $M^c(X)$.
If $k$ has characteristic~0, this is an effective geometric motive,
and if $U$ is open in $X$ with complement $Z$ there is a triangle 
$M^c(Z)\to M^c(X)\to M^c(U)$; see \cite[16.15]{MVW}. It follows that
$[M^c(X)]=[M^c(U)]+[M^c(Z)]$ in $K_0(\Meff)$. 
(This was originally proven by Gillet and Soul\'e in \cite[Thm.\,4]{GS}
before the introduction of $\DM$, but see \cite[3.2.4]{GS}.
%

\begin{Definition}
Let $K_0(Var)$ be the Grothendieck ring of varieties obtained by
imposing the relation $[U]+[X\setminus U]=[X]$ for any open $U$ in
a variety $X$.  By the above remarks, there is a well defined 
ring homomorphism $K_0(Var)\to K_0(\Meff)$.
\end{Definition}
%

Grothendieck asked in \cite[p.174]{Gletter} if this morphism was far
from being an isomorphism. We can now answer his question.

\begin{Theorem}\label{bigthm}
The homomorphism $K_0(Var)\to K_0(\Meff)$ is not an injection.
\end{Theorem}

\begin{subrem}
After this paper was posted in 2010, we were informed by J. Sebag that
Grothendieck's question had also been answered in \cite[Remark 14]{LiuSebag}.
\end{subrem}

For the proof, we need to introduce Kapranov's zeta-function.
If $X$ is any quasi-projective variety, its symmetric power $S^nX$ is the
quotient of $X^n$ by the action of the symmetric group. We define
$\zeta_{\,t}(X)=\sum [S^nX]t^n$ as a power series with coefficients in
$K_0(Var)$. 

\begin{Lemma} (\cite{Guletskii})
The following diagram is commutative:
\begin{diagram}[height=0.8cm]
  K_0(Var)&\rTo^{\zeta_{\,t}}&1+K_0(Var)[[t]]\\
\dTo<{M^c}&&\dTo>{M^c}\\
  K_0(\Meff)&\rTo^{\sigma_t}&1+K_0(\Meff)[[t]].
\end{diagram}
\end{Lemma}

\begin{proof}
It suffices to show that $[M^c(S^nX)]=\Sym^n[M^c(X)]$ in
$K_0(\Meff)$ for any $X$. This is proven by del Ba\~no 
and Navarro in \cite[5.3]{dBN}.
\end{proof}

\begin{Definition}\label{def:dr}
Following \cite[2.2]{LL04}, we say that a power series 
$f(t)=\sum r_nt^n\in R[[t]]$ is {\it determinantally rational} over 
a ring $R$ if there exists an $m,n_0>0$ such that the $m\times m$ 
Hankel 
matrices $(r_{n+i+j})_{i,j=1}^m$ have determinant~0 for all $n>n_0$.

The name comes from the classical fact (\'Emile Borel \cite{Borel}) that 
when $R$ is a field (or a domain) a power series is determinantally rational
if and only if it is a rational function $p(t)/q(t)$. For later use,
we observe that $\deg(q)<m$ and $\deg(p)<n_0$. (This is relation
($\alpha$) in \cite{Borel}.)
\end{Definition}

Clearly, if $f(t)$ is not determinantally rational over $R$ and $R\subset R'$
then $f(t)$ cannot be determinantally rational over $R'$.

As observed in \cite[2.4]{LL04}, if $f$ is a rational function in the sense
that $gf=h$ for polynomials $g(t), h(t)$ with $g(0)=1$ then $f$ is
determinantally rational. For example,
if $x=a_b$ is a finite-dimensional element of a $\lambda$-ring $R$,
with $a$ even and $b$ odd, then $\lambda_t(a)$ and 
$\lambda_t(-b)=\lambda_t(b)^{-1}$ are
polynomials so $\lambda_t(x)=\lambda_t(a)\lambda_t(b)$ and
$\sigma_t(x)=\lambda_t(x)^{-1}$ are rational functions.
This was observed by Andr\'e in \cite{Andre}.

\begin{proof}[Proof of Theorem \ref{bigthm}]
Let $X$ be the product $C\times D$ of two smooth projective curves
of genus $>0$, so that $p_g(X)>0$.
Larsen and Lunts showed in \cite[2.4, 3.9]{LL04} that $\zeta_{\,t}(X)$ 
is not determinantally rational over $R=K_0(Var)$.
On the other hand, Kimura proved in \cite{Kimura} that $X$ is a
finite-dimensional object in $\Meff$, so $\sigma_t(X)=\lambda_t(X)^{-1}$
is a determinantally rational function in $R'=K_0(\Meff)$.
It follows that $R\to R'$ cannot be an injection.
\end{proof}

\section{Symmetric functions}

We devote this section to a quick study of the ring $\Lambda$ of
symmetric functions, and especially the Schur polynomials $s_\pi$,
 referring the reader to \cite{MacD} for more information.  
In the next section, we will use these polynomials to define the
notion of Schur-finite elements in a $\lambda$-ring.

The ring $\Lambda$ is defined as the ring of symmetric ``polynomials''
in variables $\xi_i$. More precisely, it is the subring of the power series
ring in $\{\xi_n\}$ generated by $e_1=\sum \xi_n$ and the other
{\it elementary symmetric power sums} $e_i\in\Lambda$; 
if we put $\xi_r=0$ for $r>n$ then $e_i$ is the $i^{th}$ 
elementary symmetric polynomial in $\xi_1,\dots,\xi_n$; see \cite{AT}.
A major role is also played by the {\it homogeneous power sums} 
$h_n=\sum \xi_{i_1}\cdots \xi_{i_n}$
(where the sum being taken over $i_1\le\cdots\le i_n$).
Their generating functions $E(t)=\sum e_nt^n$ and $H(t)=\sum h_nt^n$
are $\prod(1+\xi_it)$ and $\prod(1-\xi_it)^{-1}$, so that
$H(t)E(-t)=1$. In fact, $\Lambda$ is a graded polynomial ring 
in two relevant ways (with $e_n$ and $h_n$ in degree $n$):
\[  \Lambda = \ZZ[e_1,...,e_n,...] = \ZZ[h_1,...,h_n,...]. \]

Given a partition $\pi=(n_1,...,n_r)$ of $n$ (so that $\sum n_i=n$), 
we let $s_\pi\in\Lambda_n$ denote the Schur polynomial of $\pi$. 
The elements $e_n$ and $h_n$ of $\Lambda$ are identified with
$s_{(1,...,1)}$ and $s_{(n)}$, respectively. The Schur polynomials 
also form a $\ZZ$-basis of $\Lambda$ by \cite[3.3]{MacD}.
By abuse, we will say that a partition $\pi$ {\it contains} a partition 
$\lambda=(\lambda_1,...,\lambda_s)$ if $n_i\ge\lambda_i$ and $r\ge s$,
which is the same as saying that the Young diagram for $\pi$ contains
the Young diagram for $\lambda$.

Here is another description of $\Lambda$, taken from \cite{Knutson}:
$\Lambda$ is isomorphic to the direct sum $R_*$ of the 
representation rings $R(\Sigma_n)$, made into a ring via
the outer product $R(\Sigma_m)\otimes R(\Sigma_n)\to R(\Sigma_{m+n})$.
Under this identification, $e_n\in \Lambda_n$ is identified with 
the class of the trivial simple representation $V_n$ of $\Sigma_n$.
More generally, $s_\pi$ corresponds to the class $[V_\pi]$ 
in $R(\Sigma_n)$ of the irreducible representation corresponding to $\pi$. 
(See \cite[III.3]{Knutson}.)

\begin{Proposition}
$\Lambda$ is a graded Hopf algebra, with coproduct $\Delta$ and
antipode $S$ determined by the formulas
\[
\Delta(e_n)=\sum_{i+j=n} e_i\otimes e_j, \quad
S(e_n)=h_n \text{ and } S(h_n)=e_n.
\]
\end{Proposition}

\begin{proof}
The graded bialgebra structure is well known and due to
Burroughs \cite{Burroughs}, who defined the coproduct on $R_*$ 
as the map induced from the restriction maps
$R(\Sigma_{m+n})\to R(\Sigma_m)\otimes R(\Sigma_n)$, and established
the formulas $\Delta(e_n)=\sum_{i+j=n} e_i\otimes e_j$.
The fact that there is a ring involution $S$ interchanging $e_n$ and $h_n$
is also well known; see \cite[I(2.7)]{MacD}. 
The fact that $S$ is an antipode does not seem to be well known, but it 
is immediate from the formula $\sum(-1)^re_rh_{n-r}=0$ of \cite[I(2.6)]{MacD}.
\end{proof}

\begin{Remark}
Atiyah shows in \cite[1.2]{Atiyah} that $\Lambda$ is isomorphic to
the graded dual $R^*=\oplus\Hom(R(\Sigma_n),\ZZ)$.
That is, if $\{ v_\pi\}$ is the dual basis in $R^n$ to the basis
$\{[V_\pi]\}$ of simple representations in $R_n$ and the restriction
of $[V_\pi]$ is $\sum c_\pi^{\mu\nu}[V_\mu]\otimes [V_\nu]$
then $v_\mu v_\nu = \sum_\pi c_\pi^{\mu\nu} v_\pi$ in $R^*$.
Thus the product studied by Atiyah on the graded dual 
$R^*$ is exactly the algebra structure dual to the coproduct $\Delta$.
\end{Remark}

\goodbreak
Let $\pi'$ denote the conjugate partition to $\pi$. 
The {\it Jacobi-Trudi identities} 
$s_\pi=\det|h_{\pi_{i}+j-i}|=\det|e_{\pi'_{i}+j-i}|$ 
show that the antipode $S$ interchanges $s_\pi$ and $s_{\pi'}$.
(Jacobi conjectured the identities, and his student Nicol\'o Trudi 
verified them in 1864; they were rediscovered by Giovanni Giambelli 
in 1903 and are sometimes called the {\it Giambelli identities}).

Let $I_{e,n}$ denote the ideal of $\Lambda$ generated by the $e_i$ with 
$i\ge n$. The quotient $\Lambda/I_{e,n}$ is the polynomial ring 
$\Lambda_{n-1}=\ZZ[e_1,...,e_{n-1}]$.
Let $I_{h,n}$ denote $S(I_{e,n})$, i.e., the ideal of $\Lambda$
generated by the $h_i$ with $i\ge n$.

\begin{Proposition}\label{Jacobi-Trudi}
The Schur polynomials $s_\pi$ for partitions $\pi$ containing 
$(1^{n})$ (i.e., with at least $n$ rows) form a $\ZZ$-basis for the 
ideal $I_{e,n}$. The Schur polynomials with at most $n$ rows form 
a $\ZZ$-basis of $\Lambda_n$.

Similarly, the Schur polynomials $s_\pi$ for partitions $\pi$ containing
$(n)$ (i.e., with $\pi_1\!\ge\! n$) form a $\ZZ$-basis for the ideal 
$I_{h,n}$.
\end{Proposition}

\begin{proof}
We prove the assertions about $I_{e,n}$; the assertion about $I_{h,n}$
follows by applying the antipode $S$. By  \cite[I.3.2]{MacD},
the $s_\pi$ which have fewer than $n$ rows project onto a $\ZZ$-basis of 
$\Lambda_{n-1}=\Lambda/I_{e,n}$. Since the $s_\pi$ form a $\ZZ$-basis 
of $\Lambda$, it suffices to show that every
partition $\pi=(\pi_1,...,\pi_r)$ with $r>n$ is in $I_{e,n}$.
Expansion along the first row of the Jacobi-Trudi identity
$s_\pi=\det|e_{\pi'_{i}+j-i}|$ shows that $s_\pi$ is in the ideal $I_{e,r}$.
\end{proof}

\begin{Corollary}\label{Imn}
The ideal $I_{h,m}\cap I_{e,n}$ of $\Lambda$ has a $\ZZ$-basis
consisting of the Schur polynomials $s_\pi$
for partitions $\pi$ containing the hook $(m,1^{n-1})=(m,1,...,1)$.
\end{Corollary}

\begin{Definition}\label{I-lambda}
For any partition $\lambda=(\lambda_1,...,\lambda_r)$, let $I_\lambda$
denote the subgroup of $\Lambda$ generated by the Schur polynomials
$s_\pi$ for which $\pi$ contains $\lambda$, i.e., $\pi_i\ge\lambda_i$
for $i=1,...,r$. We have already encountered the special cases
$I_{e,n}=I_{(1,...,1)}$ and $I_{h,n}=I_{(n)}$ in Proposition 
\ref{Jacobi-Trudi}, and 
$I_{(m,1,...,1)}=I_{h,m}\cap I_{e,n}$ in Corollary \ref{Imn}.
\end{Definition}

\begin{Example}\label{Lambda21}
Consider the partition $\lambda=(2,1)$. 
Since $I_\lambda=I_{h,2}\cap I_{e,2}$ by Corollary \ref{Imn}, 
$\Lambda_\lambda$ is the pullback of
$\ZZ[a]$ and $\ZZ[b]$ along the common quotient 
$\ZZ[a]/(a^2)=\Lambda/(I_{(1,1)}+I_{(2)})$. The universal element
of $\Lambda_{\lambda}$ is $x=(a,b)$ and if we set
$y=(0,b^2)$ then $\Lambda_{(2,1)}\cong\Z[x,y]/(y^2-x^2y)$.
Since $\lambda^n(b)=b^n$ for all $n$, it is easy to check that
$\lambda^{2i}(x)=y^i$ and $\lambda^{2i+1}(x)=xy^i$.
\end{Example}

\begin{Lemma}\label{idealI-lambda}
The $I_{\lambda}$ are ideals of $\Lambda$, and
$\{ I_\lambda\}$ is closed under intersection. 
\end{Lemma}

\begin{proof}
The Pieri rule writes $h_p s_\pi$ as a sum of $s_\mu$, where $\mu$ 
runs over partitions consisting of $\pi$ and $p$ other elements,
no two in the same column. Thus $I_\lambda$ is closed under
multiplication by the $h_p$. As every element of $\Lambda$ is a
polynomial in the $h_p$, $I_\lambda$ is an ideal.

If $\mu=(\mu_1,...,\mu_s)$ is another partition, then $s_\pi$ is in
$I_\lambda\cap I_\mu$ 
if and only if $\pi_i\ge\max\{\lambda_i,\mu_i\}$
Thus $I_\lambda\cap I_\mu=I_{\lambda\cup\mu}$.
\end{proof}

\begin{Remark}\label{I+I}
The ideal $I_\lambda + I_\mu$ need not 
be of the form $I_\nu$ for any $\nu$. 
For example, $I=I_{(2)}+I_{(1,1)}$ contains every Schur polynomial 
except $1$ and $s_1=e_1$.
\end{Remark}

\goodbreak
We conclude this section by connecting $\Lambda$ with $\lambda$-rings.
Recall from  \cite[4.4]{SGA6}, \cite[I.4]{G57}  or \cite[\S2]{AT}
that the universal $\lambda$-ring on one generator $a=a_1$ is the 
polynomial ring $\ZZ[a_1,\dots,a_n,\dots]$, with $\lambda^n(a)=a_n$.
Given an element $x$ in a $\lambda$-ring $R$, there is a unique morphism
$u_x:\Lambda\to R$ with $u_x(a)=x$.
Following \cite{Atiyah} and \cite{Knutson}, we identify
this universal $\lambda$-ring with $\Lambda$,
where the $a_i$ are identified with the $e_i\in\Lambda$.

The ring $\Lambda$ is naturally isomorphic to the ring of natural 
operations on the category of $\lambda$-rings; 
an operation $\phi$ corresponds to $\phi(a)\in\Lambda$. Conversely,
given $f\in\Lambda$, the formula $f(x)=u_x(f)$ defines
a natural operation.
The operation $\lambda^n$ corresponds to $e_n$.
The operation $\sigma^n$, defined by $\sigma^n(x)=(-1)^n\lambda^n(-x)$,
corresponds to $h_n$; this may be seen by comparing the generating 
functions $H(t)=E(-t)^{-1}$ and $\sigma_t(x)=\lambda_{-t}(x)^{-1}$.

\begin{Proposition}\label{phi(x+y)}
If $\phi$ is an element of $\Lambda$, and 
$\Delta(\phi)=\sum \phi'_i\otimes\phi''_i$ then the corresponding
natural operation on $\lambda$-rings satisfies
$\phi(x+y) = \sum \phi'_i(x)\phi''_i(y)$.
\end{Proposition}

\begin{proof}
Consider the set $\Lambda'$ of all operations in $\Lambda$ satisfying the
condition of the proposition. Since $\Delta$ is a ring homomorphism,
$\Lambda'$ is a subring of $\Lambda$. Since 
$\Delta(e_n)=\sum e_i\otimes e_{n-i}$ and 
$\lambda^n(x+y)=\sum\lambda^i(x)\lambda^{n-i}(y)$, $\Lambda'$ contains
the generators $e_n$ of $\Lambda$, and hence $\Lambda'=\Lambda$.
\end{proof}

The Littlewood-Richardson rule states that $\Delta([V_\pi])$ is a
sum $\sum c_\pi^{\mu\nu}[V_\mu]\otimes[V_\nu]$, where $\mu\subseteq\pi$
and $\pi$ is obtained from $\mu$ by concatenating $\nu$ in a certain way;
see \cite[\S I.9]{MacD}. By Proposition \ref{phi(x+y)}, we then have

\begin{Corollary}\label{LRrule}
$ s_\pi(x+y) = \sum c_\pi^{\mu\nu}s_\mu(x)s_\nu(y).$
\end{Corollary}


\section{Schur-finite $\lambda$-rings}\label{sf}

In this section we introduce the notion of a Schur-finite
element in a $\lambda$-ring $R$, and show that these elements form a subring of $R$
containing the subring of finite-dimensional elements.
We conjecture that they are the elements for which 
the virtual splitting principle holds.

\begin{Definition}\label{def:sf-element}
We say that an element $x$ in a $\lambda$-ring $R$ is {\it Schur-finite}
if there exists a partition $\lambda$ such that 
$s_\mu(x)=0$ for every partition $\mu$ containing $\lambda$. 
That is, $I_\lambda$ annihilates $x$. 
We call such a $\lambda$ a {\it bound} for $x$.
\end{Definition}
By Remark \ref{I+I}, 
$x\in R$ may have no unique minimal bound $\lambda$.
By Example \ref{lambda2-3} below, $s_\lambda(x)=0$ does not imply that
$\lambda$ is a bound for $x$.

\begin{Proposition}\label{sf-ideal}
Each $I_\lambda$ is a radical $\lambda$-ideal, and
$\Lambda_\lambda=\Lambda/I_\lambda$ is a reduced $\lambda$-ring. 
Thus every Schur-finite $x\!\in\! R$ with bound $\lambda$ determines 
a $\lambda$-ring map $f:\Lambda_\lambda\!\to\!R$ with $f(a)=x$.

Moreover, if $\lambda$ is a rectangular partition then $I_\lambda$ is
a prime ideal, and $\Lambda_\lambda$ is a subring of a polynomial ring
in which $a$ becomes finite-dimensional. 

In general, $\Lambda_\lambda$
is a subring of $\prod\Lambda_{\beta_i}$ and hence of a product of
polynomial rings in which $a$ becomes finite-dimensional.
\end{Proposition}

Proposition \ref{sf-ideal} verifies Conjecture 3.9 of \cite{KKT}.

\begin{proof}
Fix a rectangular partition $\beta=((m+1)^{n+1})=(m+1,...,m+1)$, and
set $a=\sum_1^m a_i$, $b=\sum_1^n b_j$.
Consider the universal $\lambda$-ring map
\[
f:\Lambda\to\Lambda_m\otimes\Lambda_{-n}\cong\ZZ[a_1,...,a_m,b_1,...,b_n]
\]
sending $e_1$ to the finite-dimensional element $a+b$ 
(see Definition \ref{free-even}). 
We claim that the kernel of $f$ is $I_\beta$. 
Since $\ker f$ is a $\lambda$-ideal, this proves that $I_\beta$ 
is a $\lambda$-ideal and that $\Lambda/I_\beta$ embeds into
the polynomial ring $\ZZ[a_1,...,a_m,b_1,...,b_n]$.
%
Since any partition $\lambda$ can be written as a union of rectangular 
partitions $\beta_i$, Lemma \ref{idealI-lambda} implies that
$I_\lambda=\cap I_{\beta_i}$ is also a $\lambda$-ideal.

By the Littlewood-Richardson rule \ref{LRrule}, 
$f(s_\pi)=s_\pi(a+b)=\sum c_\pi^{\mu\nu}s_\mu(a)s_{\nu}(b)$, where 
$\mu$ and $\nu$ run over all partitions such that
$\pi$ is obtained from $\mu$ by concatenating $\nu$ in a certain way.
We may additionally restrict the sum to $\mu$ with at most $m$ rows and $\nu$
with $\nu_1\le m$, since otherwise $s_\mu(a)=0$ or $s_\nu(b)=0$.
By Proposition \ref{Jacobi-Trudi}, the $s_\mu(a)$
run over a basis of $\Lambda_m$ and the $s_\nu(b)$ run over a basis of
$\Lambda_{-n}$.

If $\pi$ contains $\beta$ then $f(s_\pi)=s_\pi(a+b)=0$, because in 
every term of the above expansion, either 
the length of $\mu$ is $>m$ or else $\nu_1>n$.
Thus $I_\beta\subseteq\ker f$.

For the converse, we use the reverse lexicographical ordering of
partitions \cite[p.\,5]{MacD}. For each $\pi$ not containing $\beta$, 
set $\mu_\pi=(\pi_1,\dots,\pi_m)$; this is the maximal $\mu$ (for this
ordering) such that $c^{\mu\nu}_\pi\ne0$ (with $\nu_\pi=\pi-\mu_\pi$).
Given $t=\sum_{\beta\not\subseteq\pi} d_\pi s_\pi$, 
choose $\mu$ maximal subject to $\mu=\mu_\pi$ for some $\pi$ with
$d_\pi\ne0$; choose $\pi$ maximal with $\mu=\mu_\pi$ and $d_\pi\ne0$,
and set $\nu=\nu_\pi$. 
Then the coefficient of $s_\mu(a)s_\nu(b)$ in $f(t)$ is $d_\pi\ne0$.
Thus $\ker f\subseteq I_\beta$.
%
\end{proof}


\begin{Corollary}\label{lambda22}
$\Lambda_{(2,2)}$ is the subring $\Z+x\Z[a,b]$
of $\Z[a,b]$, where $x=a+b$.
\end{Corollary}

\begin{proof} 
By Proposition \ref{sf-ideal}, $\Lambda_{(2,2)}$ is the subring of
$\Z[a,b]$ generated $x=a+b$ and the $\lambda^n(x)$. Since
\[
\lambda^{n+1}(x)=a\lambda^n(b)+\lambda^{n+1}(b)=ab^n+b^{n+1}=xb^n,
\]
we have $\Lambda_{(2,2)}=\Z[x,xb,xb^2,\dots,xb^n,\dots]=\Z+x\Z[a,b]$.
\end{proof}

\begin{subrem}
The ring $\Lambda_{(2,2)}$ was studied in \cite[3.8]{KKT}, where it was shown 
that $\Lambda_{(2,2)}$ embeds into $\Z[x,y]$ sending $e_n$ to $xy^{n-1}$.
This is the same as the embedding 
in Corollary \ref{lambda22},
up to the change of coordinates $(x,y)\!=\!(a+b,b)$.
\end{subrem}

\begin{Example}\label{lambda2-3}
Let $I$ be the ideal  of $\Lambda_{(2,2)}$ generated by the $\lambda^{2i}(x)$
($i>0$) and set $R=\Lambda_{(2,2)}/I$. Then $R$ is a $\lambda$-ring 
and $x$ is a non-nilpotent element such that
$\lambda^{2i}(x)=0$ but $\lambda^{2i+1}(x)\ne0$. In particular,
$\lambda^2(x)=0$ yet $\lambda^3(x)\ne0$.

To see this, we use the embedding of Corollary \ref{lambda22} to see that 
$I$ contains $x(xb^{2i-1})$ and $(xb)(xb^{2i-1})$ and hence
the ideal $J$ of $\Z[a,b]$ generated by $x^2b$.
In fact,  $I$ is additively generated by $J$ and the $\{xb^{2i-1}\}$.
It follows that $R$ has basis $\{1,x^n, xb^{2n}|n\ge 1\}$. 
Since $\lambda^n(\lambda^{2i}(x))$ is equivalent 
to $\lambda^{2in}(x)=xb^{2in-1}$ modulo $J$ (by \ref{scholium}), 
it lies in $I$. Hence $I$ is a $\lambda$-ideal of $\Lambda_{(2,2)}$.

There is no $\lambda$-ring extension $R\subset R'$ in which 
$x=\ell_1-\ell_2$ for line elements $\ell_i$, 
because we would have $\lambda^3(x)=\lambda^3(x+\ell_2)=0$.
On the other hand, there is a $\lambda$-ring extension $R\subset R'$
in which $x=\ell_1+\ell_2-\ell_3-\ell_4$ for line elements $\ell_i$.
\end{Example}

\begin{Lemma}\label{Lx+y}
If $x$ and $y$ are Schur-finite, so is $x+y$.
\end{Lemma}

\begin{proof}
Given a partition $\lambda$, there is a partition $\pi_0$ such that 
whenever $\pi$ contains $\pi_0$, one of the partitions $\mu$ and $\nu$ 
appearing in the Littlewood-Richardson rule \ref{LRrule} must contain 
$\lambda$.
If $x$ and $y$ are both killed by all Schur polynomials indexed by
partitions containing $\lambda$, we must therefore have $s_\pi(x+y)=0$.
\end{proof}

\begin{Corollary}\label{finite-is-Schur}
Finite-dimensional elements are Schur-finite.
\end{Corollary}

\begin{proof}
Proposition \ref{Jacobi-Trudi} shows that even and odd elements
are Schur-finite.
\end{proof}

\begin{Example}
If $R$ is a binomial ring containing $\QQ$, then
every Schur-finite element is finite-dimensional.
This follows from Example \ref{binom} and  \cite[Ex.\,I.3.4]{MacD}, which says 
that $s_\pi(r)$ is a rational number times a product of terms $r-c(x)$,
where the $c(x)$ are integers.
\end{Example}

\begin{Example}\label{ex:sf-not-finite}
The universal element $x$ of the $\lambda$-ring $\Lambda_{(2,1)}$ is
Schur-finite but not finite-dimensional. To see this, recall from
Example \ref{Lambda21} that $\Lambda_{(2,1)}\cong\Z[x,y]/(y^2-x^2y)$.
Because $\Lambda_{(2,1)}$ is graded, if $x$ were finite-dimensional
it would be the sum of an even and odd element in the degree~1 part
$\{nx\}$ of $\Lambda_{(2,1)}$.
If $n\in\NN$, $nx$ cannot be even because the second 
coordinate of $\lambda^k(nx)$ is $\chose{-n}{k}\, b^k$  by \ref{free-even}.
And $nx$ cannot be odd, because the first coordinate of $\sigma^k(nx)$ is
$(-1)^k\chose{-n}{k}\,a^k$.
\end{Example}

\begin{Lemma}\label{schur-descent}
Let $R\subset R'$ be an inclusion of $\lambda$-rings.
If $x\in R$ then $x$ is Schur-finite in $R'$, 
if and only if $x$ is Schur-finite in $R$.
In particular, if $x$ is finite-dimensional in $R'$, 
then $x$ is Schur-finite in $R$.
\end{Lemma}

\begin{proof}
Since $s_\pi(x)$ may be computed in either $R$ or $R'$,
the set of partitions $\pi$ for which $s_\pi(x)=0$ is the same for $R$ and $R'$.
The final assertion follows from Lemma \ref{finite-is-Schur}.
\end{proof}

\begin{Lemma}\label{L-x}
If $\pi$ is a partition of $n$, $s_{\pi'}(-x)=(-1)^n s_{\pi}(x)$.
\end{Lemma}

\begin{proof}
Write $s_\pi$ as a homogeneous polynomial $f(e_1,e_2,...)$
of degree $n$. Applying the antipode $S$ in $\Lambda$, we have $s_{\pi'}=f(h_1,h_2,...)$.
It follows that $s_{\pi'}(-x)=f(\sigma^1,\sigma^2,...)(-x)$.
Since $\sigma^i(-x)=(-1)^i\lambda^i(x)$, and $f$ is homogeneous, we have
\[
s_{\pi'}(-x)=f(-\lambda^1,+\lambda^2,...)(x) =
(-1)^n f(\lambda^1,\lambda^2,...)(x) = s_\pi(x).\qedhere
\]
\end{proof}

\begin{subrem}
If $a$ is a line element then $s_\pi(ax)=a^n s_\pi(x)$.
From Lemma \ref{schur-descent}, we have
$s_\pi(-ax)=(-a)^n s_{\pi'}(x)$.
\end{subrem}

\begin{Theorem}
The Schur-finite elements form a subring of any $\lambda$-ring,
containing the subring of finite-dimensional elements.
\end{Theorem}


\begin{proof} The Schur-finite elements are closed under addition by
Lemma \ref{Lx+y}. Since $\pi$ contains $\lambda$ just in case 
$\pi'$ contains $\lambda'$, Lemma \ref{L-x} implies that $-x$ is 
Schur-finite whenever $x$ is. Hence the Schur-finite elements form a 
subgroup of $R$. It suffices to show that if $x$ and $y$ are Schur-finite
in $R$, then $xy$ and all $\lambda^i(x)$ are Schur-finite.

Let $x$ be Schur-finite with rectangular bound $\mu$, so there is a map from the
$\lambda$-ring $\Lambda_\mu$ to $R$ sending the generator
$e$ to $x$. Embed $\Lambda_\mu$ in $R'=\Z[a_1,\ldots,b_1,\ldots]$
using Proposition \ref{sf-ideal}. Since every element of $R'$ 
is finite-dimensional, $\lambda^n(e)$ is finite-dimensional
in $R'$, and hence Schur-finite in $\Lambda_\mu$ by Lemma \ref{schur-descent}.
It follows that the image $\lambda^n(x)$ of $\lambda^n(e)$ in $R$ is also Schur-finite.

Let $x$ and $y$ be Schur-finite with rectangular bounds $\mu$ and
$\nu$, and let $\Lambda_\mu\to R$ and $\Lambda_\nu\to R$ be the
$\lambda$-ring maps sending the generators $e_\mu$ and $e_\nu$ to $x$
and $y$.  Since the induced map $\Lambda_\mu\tens \Lambda_\nu\to R$
sends $e_\mu\tens e_\nu$ to $xy$, we only need to show that
$e_\mu\tens e_\nu$ is Schur-finite. But 
$\Lambda_\mu\tens \Lambda_\nu\subset \Z[a_1,\ldots,b_1,\ldots]\tens
\Z[a_1,\ldots,b_1,\ldots]$, and in the larger ring every element
is finite-dimensional, including the tensor product. By Lemma \ref{schur-descent},
$e_\mu\tens e_\nu$ is Schur-finite in $\Lambda_\mu\tens \Lambda_\nu$.
\end{proof}

\begin{Conjecture}[Virtual Splitting principle]\label{VSP}
Let $x$ be a Schur-finite element of a $\lambda$-ring $R$. Then
$R$ is contained in a larger $\lambda$-ring $R'$ such that
$x$ is finite-dimensional in $R'$, i.e., there are
line elements $\ell_i$, $\ell'_j$ in $R'$ so that
\[ x=(\sum\ell_i) - (\sum\ell'_j). \] 
\end{Conjecture}

\begin{Example}
The virtual splitting principle holds in the universal case, 
where $R_0=\Lambda_\beta$. Indeed, we know that
$x$ is $\sum a_i+\sum b_j$ in $R_0'=\Z[a_1,\ldots,b_1,\ldots]$.
Since $\ell_j=-b_j$ is a line element, $x$ is a difference of sums
of line elements in $R_0'$. 

Unfortunately, although the induced map $f:R\to R\tens_{R_0}R_0'$ sends 
a Schur-finite element $x$ to a difference of sums of line elements,
the map $f$ need not be an injection.
\end{Example}

\begin{Proposition}\label{VSP-fields}
If a $\lambda$-ring $R$ is a domain, $R$ is contained in a $\lambda$-ring
$R'$ such that every Schur-finite element of $R$ is a a difference
of sums of line elements in $R'$.
\end{Proposition}

\begin{proof}
Let $E$ denote the algebraic closure of the fraction field of $R$ 
and set $R'=W(E)$;
$R$ is contained in $R'$ by $R\TTo{\lambda_t} W(R)\subset W(E)$.
If $x\in R$ is Schur-finite then $\lambda_t(x)$ is determinentally
rational in $E[[t]]$ and hence a rational function $p/q$ in $E(t)$
(see \ref{def:dr}).
Factoring $p$ and $q$ in $E[t]$, we have 
$$\lambda_t(x)=\prod(1-\alpha_i t)/\prod(1-\beta_j t)$$
for suitable elements $\alpha_i, \beta_j$ of $E$.
Since the underlying abelian group of $W(E)$ is $(1+tE[[t]],\times)$ and
the $\ell_i=(1-\alpha_i t)$ and $\ell'_j=(1-\beta_j t)$ 
are line elements in $W(E)$, we are done.
\end{proof}

The proof shows that a bound $\pi$ on $x$ determines a bound on 
the degrees of $p(t)$ and $q(t)$ and hence on
the number of line elements $\ell_i$ and $\ell'_j$ in the virtual sum.

\begin{Corollary}
The virtual splitting principle holds for reduced $\lambda$-rings.
\end{Corollary}

\begin{proof}
Let $R$ be a reduced ring. If $P$ is a minimal prime of $R$ then 
the localization $R_P$ is a domain and 
$R$ embeds into the product $\prod E_P$ of the algebraic closures of the
fields of fractions of the $R_P$. If in addition $R$ is a $\lambda$-ring
then $R$ embeds into the $\lambda$-ring $R'=\prod W(E_P)$.
If $x$ is Schur-finite in $R$ with bound $\pi$ then $\lambda_t(R)$ is
determinantly rational and each factor of $\lambda_t(x)$
is a rational function in $E_P(t)$; the bound $\pi$ determines a bound
$N$ on the degrees of the numerator and denominator in each component.
By Theorem \ref{VSP-fields}, there are line elements
$\ell_1,\dots, \ell_{N}, \ell'_1,\dots,\ell'_{N}$ in each component
so that $x=(\sum \ell_i)-(\sum \ell'_j)$ in $R'$.
\end{proof}

As more partial evidence for Conjecture \ref{VSP}, we show that the 
virtual splitting principle holds for elements bounded by the hook $(2,1)$.

\begin{Theorem}\label{21-splitting}
Let $x$ be a Schur-finite element in a $\lambda$-ring $R$.
If $x$ has bound $(2,1)$, then $R$ is contained in a $\lambda$-ring $R'$
in which $x$ is a virtual sum $\ell_1+\ell_2-a$ of line elements.
\end{Theorem}

\begin{proof}
The polynomial ring $R[a]$ becomes a $\lambda$-ring once we declare
$a$ to be a line element. Set $y=x+a$, and let $I$ be the ideal of $R[a]$
generated by $\lambda^3(y)$. 

For all $n\ge2$, the equation $s_{n,1}(x)=0$ 
yields $\lambda^{n+1}(x)=x\lambda^n(x)=x^{n-1}\lambda^2(x)$ in $R$, and
therefore $\lambda^{n+1}(y)=(a+x)x^{n-2}\lambda^2(x)=x^{n-2}\lambda^3(y)$.
It follows from Scholium \ref{scholium} that $\lambda^m(\lambda^3y)\in I$
for all $m\ge1$ and hence that 
\[
\lambda^n(f\cdot\lambda^3y) = P_n(\lambda^1(f),\dots,\lambda^n(f);
           \lambda^1(\lambda^3y),\dots,\lambda^n(\lambda^3y))
\]
is in $I$ for all $f\in R[a]$. Thus $I$ is a $\lambda$-ideal of $R[a]$,
$A=R[a]/I$ is a $\lambda$-ring, and the image of $y$ in $A$ is even of
degree~2. By the Splitting Principle \ref{finite-splitting}, the image of 
$x=y-a$ is a virtual sum $\ell_1+\ell_2-a$ of line elements
in some $\lambda$-ring containing $A$.

To conclude, it suffices to show that $R$ injects into $A=R[a]/I$.
If $r\in R$ vanishes in $A$ then $r=f\lambda^3(y)$ for some
$f=f(a)$ in $R[a]$. We may take $f$ to have minimal degree $d\ge0$.
Writing $f(a)=c\,a^d+g(a)$, with $c\in R$ and $\deg(g)<d$,
the coefficient of $a^{d+1}$ in $f\lambda^3(y)$, namely $c\,\lambda^2(x)$,
must be zero. But then $c\lambda^3y=0$, and $r=g\,\lambda^3y$,
contradicting the minimality of $f$.
\end{proof}

\begin{Remark}
The rank of a Schur-finite object with bound $\pi$ cannot be well defined
unless $\pi$ is a rectangular partition. This is because any
rectangular partition $\mu=(m+1)^{n+1}$ contained in $\pi$ yields a
map $R\to R'$ sending $x$ to an element of rank $m-n$. If
$\pi$ is not rectangular there are different maximal rectangular
sub-partitions with different values of $m-n$.

For example, let $x$ be the element of Theorem \ref{21-splitting}.
By Lemma \ref{L-x}, $-x$ also has bound $(2,1)$. Applying Theorem
\ref{21-splitting} to $-x$ shows that $R$ is also contained in a 
$\lambda$-ring $R''$ in which $x$ is a virtual sum $a-\ell_1-\ell_2$ 
of line bundles. Therefore $x$ has rank~1 in $R'$, and has rank $-1$ in $R''$.
\end{Remark}

\medskip
\section{Rationality of $\lambda_t(x)$}

Let $R$ be a $\lambda$-ring and $x\in R$. One central question is to 
determine when the power series $\lambda_t(x)$ is a rational function.
(See \cite{Andre}, \cite{LL04}, \cite{Heinloth}, \cite{Guletskii}, 
\cite{B1, B2}, \cite{KKT} for example.) Following \cite[2.1]{LL04},
we make this rigorous by restricting to power series in $R[[t]]$
congruent to $1$ modulo $t$ and define a (globally) rational function
to be a power series $f(t)$ such that there exist polynomials 
$p,q\in R[t]$ with $p(0)=q(0)=1$ such that $p(t)=f(t)q(t)$.

As noted in \ref{def:dr}, it is well known that if $x$ is a
finite-dimensional element then $\lambda_t(x)$ is a rational function.
Larsen and Lunts observed in \cite{LL04} that the property of being a 
rational function is not preserved by passing to subrings and proposed
replacing `rational function' by `determinantally rational function'
(see \ref{def:dr}). We propose an even weaker condition, which we now define.

Given a power series $f(t)=\sum r_nt^n\in R[[t]]$ and a partition $\pi$,
we form the the Jacobi-Trudi matrix $(a_{ij})$ with $a_{i,j} = r_{\pi'_i+i-j}$
and define $s_{\pi}(f)\in R$ to be its determinant. (If $\pi$ has $m$ columns,
$\pi'$ has $m$ rows and $(a_{ij})$ is an $m\times m$ matrix over $R$.) 
The terminology comes from the fact that the commutative ring homomorphism
$\rho:\Lambda\to R$, defined by $\rho(x_n)= r_n$, satisfies
$\rho(s_\pi)=\det(a_{i,j})$ by the Jacobi-Trudi identities.

\begin{Definition}
Let $R$ be a commutative ring. We say that a power series 
$f(t)=\sum r_nt^n\in R[[t]]$ is {\it Schur-rational} over $R$ 
if there exists a partition $\mu$ such that $s_\pi(f)=0$
for every partition $\pi$ containing $\mu$.
\end{Definition}

If $\mu$ is a rectangular partition then $(a_{i,j})$ is the matrix
$(r_{n+i-j})$ in Definition \ref{def:dr} up to row permutation.
It follows that if $f(t)$ is Schur-rational then it is determinentally 
rational. The converse fails, as we show in Example \ref{ex:Sr}.

It is easy to see that a (globally) rational function is Schur-rational.
Thus being Schur-rational is a property of $f$
intermediate between being rational and being determinentally rational.

\begin{Example}\label{ex:Sr}
Let $R_m$ be the quotient of $\Lambda$ by the ideal generated by all
$m$-fold products $x_{i_1}\cdots x_{i_m}$ where $|i_j-i_k|<2m$ for all $j,k$.
Then $f(t)=\sum x_n t^n$ is determinentally rational. On the other hand,
$f(t)$ is not Schur-rational because for each $\lambda$ with $l$ rows
there are lacunary partitions $\pi=(\pi_1,\pi_2,\dots,\pi_l)$ (meaning that
$\pi_1\gg\pi_2\gg\cdots\gg\pi_l\gg0$) containing $\lambda$ which are
nonzero in $R_m$, because $s_\pi(f)$ is an alternating sum of monomimals
and the diagonal monomial $\prod r_{\pi_i}$ is nonzero
and occurs exactly once.
\end{Example}

The notion of Schur-rationality is connected to Schur-finiteness.

\begin{Proposition}\label{sfdetrat}
An element $x$ in a $\lambda$-ring is Schur-finite if and only if 
the power series $\lambda_t(x)$ is Schur-rational.

In particular, if $x$ is Schur-finite then $\lambda_t(x)$
is determinantally rational.
%
\end{Proposition}

The ``if'' part of this proposition was proven in \cite[3.10]{KKT} for
$\lambda$-rings of the form $K_0(\calA)$, using categorical methods. 

\begin{proof}
By definition, the power series $\lambda_t(x)$ is Schur-rational 
if and only if there is a partition $\mu$ so that for every $\pi$ containing 
$\mu$, the determinant $\det(\lambda^{\pi'_i+i-j}(x))$ is zero.
Since this determinant is $s_\pi(x)$ by the Jacobi-Trudi identity,
this is equivalent to $x$ being Schur-finite (definition \ref{def:sf-element}).
\end{proof}
%

We conclude by connecting our notion of Schur-finiteness to the
notion of a Schur-finite object in a $\QQ$-linear tensor category
$\calA$, given in \cite{Mazza}).  By definition, an object $A$ is
{\it Schur-finite} if some $S_\lambda(A)\cong0$ in $\calA$.
By \cite[1.4]{Mazza}, this implies that $S_\pi(A)=0$ for all
$\pi$ containing $\lambda$.
%
It is evident that if $A$ is a  Schur-finite object of $\calA$ then
$[A]$ is a Schur-finite element of $K_0(\calA)$.
However, the converse need not hold. 
For example, if $\calA$ contains infinite direct sums then $K_0(\calA)=0$
by the Eilenberg swindle, so $[A]$ is always Schur-finite.

Here are two examples of Schur-finite objects whose class in $K_0(\calA)$
is finite-dimensional even though they are not finite-dimensional objects. 

\begin{Example}
Let $\calA$ denote the abelian category of positively graded modules
over the graded ring $A=\QQ[\varepsilon]/(\varepsilon^2=0)$.
It is well known that $\calA$ is a tensor category under $\otimes_{\QQ}$,
with the $\lambda$-ring $K_0(\calA)\cong\Lambda_{-1}=\Z[b]$; $1=[Q]$ and 
$b=[\QQ[1]]$. 
The graded object $A$ is Schur-finite but not finite-dimensional in $\calA$
by \cite[1.12]{Mazza}. However, $[A]$ is a finite-dimensional element in 
$K_0(\calA)$ because $[A]=[\QQ]+[\QQ[1]]$.
\end{Example}



\begin{Example}[O'Sullivan]
Let $X$ a Kummer surface; then there is an open subvariety $U$ of $X$,
whose complement $Z$ is a finite set of points, such that $M(U)$ is 
Schur-finite but not finite-dimensional in the Kimura-O'Sullivan sense 
in the category $\calM$ of motives
\cite[3.3]{Mazza}. However, it follows from the distinguished triangle
\[ 
M(Z)(2)[3]\to M(U)\to M(X)\to M(Z)(2)[4]
\]
that $[M(U)]=[M(Z)(2)[3]]+[M(X)]$ in $K_0(\DM_{gm})$ and hence in 
$K_0(\calM)$. Since both $M(X)$ and $M(Z)(2)[3]$ are finite-dimensional, 
$[M(U)]$ is a finite-dimensional element of $K_0(\calM)$.
\end{Example}

\begin{Proposition}
Let $M$ be a classical motive. If $M$ is Schur-finite in $\calM$, 
then $\lambda_t([M])$ is determinantally rational. 
If $\lambda_t([M])$ is determinantally rational, then there exists a 
partition $\lambda$ such that $S_\lambda(M)$ is a phantom motive.
\end{Proposition}


\begin{proof}
If $M$ is Schur-finite, then there is a $\lambda$ such that
$0=[S_\pi M]=s_\pi([M])$ for all $\pi\supseteq\lambda$. 
Thus $[M]$ is Schur-finite in $K_0(\calM)$ or equivalently,
by \ref{sfdetrat}, $\lambda_t([M])$ is determinantally rational.
If $\lambda_t([M])$ is Schur-finite, with bound $\lambda$, then
$0=s_\lambda([M])=[S_\lambda M]$  in $K_0(\Meff)$.
By Proposition \ref{m0phantom}, $S_\lambda(M)$ is a phantom motive.
%
\end{proof}


\hyphenation{Chris-tophe}
\subsection*{Acknowledgments}
The authors would like to thank Anders Buch, Alessio Del Padrone 
and Christophe Soul\'e for valuable discussions.


\end{document}

%% file: macroMotives.tex
\DeclareMathAlphabet{\mathbbb}{U}{bbold}{m}{n}
\DeclareMathAlphabet{\mathmanual}{U}{manfnt}{m}{n}

\newcommand{\calA}{\mathcal{A}}

\newcommand{\calM}{\mathcal{M}}

\newcommand{\tens}{\otimes}

\newcommand{\Hom}{\mathrm{Hom}}

\newcommand{\DM}{\mathbf{DM}}

\newcommand{\Q}{\mathbb{Q}}

\newcommand{\isom}{\cong}

\def\ch{\mathop{\mathcal Ch}\nolimits}
\def\cch{\mathop{\mathcal CCh}\nolimits}

\def\ch-{\mathop{{\mathcal Ch}_-}\nolimits}
\def\cch-{\mathop{{\mathcal CCh}_-}\nolimits}
\def\ch+{\mathop{{\mathcal Ch}_+}\nolimits}
\def\cch+{\mathop{{\mathcal CCh}_+}\nolimits}

\def\det{\mathop{\rm det}\nolimits}

\def\ker#1{\mathop{{\rm Ker}(#1)}\nolimits}

\def\NN{{\mathbb N}}
\def\ZZ{{\mathbb Z}}
\def\QQ{{\mathbb Q}}

\def\fraz#1 #2 {\frac{#1}{#2}}

\def\h#1{{\kern .07em}^h\kern-.04em{#1}}
\def\a#1{{\kern .07em}{#1}_a}

\def\TTo#1{\mathop{\longrightarrow}\limits^{#1}}

\newcommand{\BD}{\begin{definition}}
\newcommand{\BDS}{\begin{definitions}}
\newcommand{\ED}{\end{definition}}
\newcommand{\EDS}{\end{definitions}}

\newcommand{\BR}{\begin{remark}}
\newcommand{\BRS}{\begin{remarks}}
\newcommand{\ER}{\end{remark}}
\newcommand{\ERS}{\end{remarks}}

\newcommand{\BN}{\begin{notation}}
\newcommand{\BNS}{\begin{notations}}
\newcommand{\EN}{\end{notation}}
\newcommand{\ENS}{\end{notations}}

\newcommand{\BE}{\begin{example}}
\newcommand{\BES}{\begin{examples}}
\newcommand{\EE}{\end{example}}
\newcommand{\EES}{\end{examples}}

\newcommand{\BT}{\begin{theorem}}
\newcommand{\ET}{\end{theorem}}

\newcommand{\BP}{\begin{proposition}}
\newcommand{\EP}{\end{proposition}}

\newcommand{\BC}{\begin{corollary}}
\newcommand{\EC}{\end{corollary}}

\newcommand{\BL}{\begin{lemma}}
\newcommand{\EL}{\end{lemma}}

\newcommand{\Bel}{\begin{description}}
\newcommand{\Eel}{\end{description}}

\def\ch{\mathop{\mathcal Ch}\nolimits}
\def\cch{\mathop{\mathcal CCh}\nolimits}

\def\ch-{\mathop{{\mathcal Ch}_-}\nolimits}
\def\cch-{\mathop{{\mathcal CCh}_-}\nolimits}
\def\ch+{\mathop{{\mathcal Ch}_+}\nolimits}
\def\cch+{\mathop{{\mathcal CCh}_+}\nolimits}

\def\det{\mathop{\rm det}\nolimits}

\def\ker#1{\mathop{{\rm Ker}(#1)}\nolimits}

\def\NN{{\mathbb N}}
\def\ZZ{{\mathbb Z}}
\def\QQ{{\mathbb Q}}

\DeclareMathAlphabet{\mathbbb}{U}{bbold}{m}{n}

\def\fraz#1 #2 {\frac{#1}{#2}}

\def\h#1{{\kern .07em}^h\kern-.04em{#1}}
\def\a#1{{\kern .07em}{#1}_a}

\def\TTo#1{\mathop{\longrightarrow}\limits^{#1}}
